\documentclass{amsart}
\theoremstyle{plain}
\newtheorem{thm}{Theorem}[section]
\newtheorem{lem}[thm]{Lemma}
\newtheorem{prop}[thm]{Proposition}
\newtheorem{cor}[thm]{Corollary}
\newtheorem{conj}[thm]{Conjecture}
\theoremstyle{definition}
\newtheorem{defn}[thm]{Definition}
\theoremstyle{remark}
\newtheorem{rem}[thm]{Remark}

\newcommand{\lag}{\langle}
\newcommand{\rag}{\rangle}
\newcommand{\aaa}{\alpha}
\newcommand{\dis}{\displaystyle}

\newcommand{\GG}{\mathcal{G}}
\newcommand{\ep}{\epsilon}
\newcommand{\fa}{\forall}
\newcommand{\ti}{\tilde}
\newcommand{\mF}{\mathscr{F}}

\newcommand{\mE}{\mathscr{E}}

\newcommand{\mL}{\mathscr{L}}
\newcommand{\cu}{\mathcal{U}}

\newcommand{\ot}{\otimes}

\newcommand{\mO}{\mathcal{O}}

\newcommand{\form}{\langle,\rangle}
\newcommand{\ff}{\underline{\langle,\rangle}}
\usepackage{amsfonts}
\usepackage{pdflscape}
\usepackage{amsmath,amssymb,latexsym,amsthm}
\usepackage{mathtools}
\usepackage{graphicx}
\usepackage{mathrsfs}
\usepackage{amscd}
\DeclareMathOperator\rk{rank}

\DeclareMathOperator\ho{Hom}
\DeclareMathOperator\M{MSG}
\DeclareMathOperator\dv{div}
\DeclareMathOperator\Gr{Gr}
\DeclareMathOperator\supp{Supp}
\DeclareMathOperator\s{span}
\DeclareMathOperator\res{res}
\DeclareMathOperator\Id{Id}
\input xy
\xyoption{all}
\begin{document}
\title{Expected Dimensions of Higher-rank Brill-Noether Loci}
\author{Naizhen Zhang}
\address{Naizhen Zhang, Department of Mathematics, University of California, Davis, One Shields Avenue, Davis, CA 95616, USA}
\email{nzhzhang@math.ucdavis.edu}
\maketitle
\begin{abstract}
In this paper, we prove a new expected dimension formula for certain rank two Brill-Noether loci with fixed special determinant. This answers a question asked by Osserman and also leads to a new and much simpler proof of a theorem in \cite{Ofix}. Our result generalizes the well-known result by Bertram, Feinberg and independently Mukai on expected dimension of rank two Brill-Noether loci with canonical determinant and partially verifies a conjecture (in rank two) of Grzegorczyk and Newstead on coherent systems. 
\end{abstract}
\section{Introduction}
Brill-Noether theory studies the moduli space of vector bundles of fixed rank, with a given amount of global sections, over
algebraic curves with given genus. Based on the theory of determinantal varieties, the expected dimension of the moduli stack
of vector bundles on a genus $g$ curve, of rank $r$ and fixed determinant $\mL$, together with a
$k$-dimensional space of sections is 
$$\rho(r,d,k,g)=(r^2-1)(g-1)-k(k-d+r(g-1)).$$ 
When $\mL=\omega$ is the canonical line bundle and the rank of
the vector bundles is two, this number simplifies to $3g-3-k^2$. However, utilizing the symplectic form induced by the
determinant map, Bertram, Feinberg in \cite{Ber} and Mukai in \cite{MCan} obtained asymptotically much larger expected
dimension in this particular case, namely, a lower bound for the dimension of the corresponding moduli stack which is
$3g-3-\binom{k+1}{2}$. For a wide range of $(g,k)$, this is known to be a sharp lower bound. (See \cite{M1}, \cite{LNP} \cite{Zhang1}.)

More generally, when the determinant $\mL$ is special and $m=h^1(\mL)>1$, there are $m$ independent maps from $\mL$ to the $\omega$.
This further induces $m$ independent symplectic forms on the space of sections (as compared to \textit{the} symplectic form in
the canonical determinant case). In \cite{Ofix} and \cite{Osp2}, Osserman generalizes the idea of Bertram, Feinberg and Mukai
to obtain new expected dimensions for some more general cases with special determinant. There, main machinery is the theory
of multiply symplectic Grassmannians, which plays a similar role to the theory of symplectic Grassmannians in the canonical
determinant case.

Due to the existence of more than one symplectic form, a multiply symplectic Grassmanian (unlike a symplectic Grassmannian) is
not necessarily smooth: it is the intersection of various symplectic Grassmannians inside an ambiance Grassmannian and one needs to analyze whether these symplectic Grassmannians intersect transversely. More concretely, we shall analyze when the tangent
space of a multiply symplectic Grassmannian $MSG(k,E,\form_1,...,\form_m)$ has the expected dimension $k(n-k)-m\binom{k}{2}$.

After re-formulating the question, we first give a new and much simpler proof of the following theorem, which is a key technical result in \cite{Ofix} (see Proposition 4.6, \cite{Ofix}):
\begin{thm}\label{prop:tech}
   Consider the doubly symplectic Grassmannian, $MSG(k,E,\form_{1,2})$. The tangent space at a point $[V]$ has dimension
   $$\dim(T_VMSG(k,E,\form_{1,2}))=k(n-k)-2\binom{k}{2}$$ if and only if for every linear combination $\form$ of
   $\form_1,\form_2$, there does not exist a 2-dimensional subspace $V'$ of $V$ such that $\form|_{V'\times E/V}=0$.
\end{thm}
This leads to one of the main results in \cite{Ofix}:
\begin{thm}\label{thm1}Let $C$ be a smooth projective curve of genus $g$, and $\mL$ a line bundle of degree $d$ on $C$. Denote $\GG(2,\mL,k)$
to be the moduli space of the following data: a rank two vector bundle of degree $d$ (with determinant $\mL$), together with a
$k$-dimensional space of global sections. Suppose that $h^1(C,\mL)\ge 2$. Then every irreducible component of $\GG(2,\mL,k)$ has
dimension at least
$$\rho^2_{\mL}(k,g):=\rho(2,d,k,g)-g+2\binom{k}{2}.$$
\end{thm}
As pointed out by Osserman in \cite{Ofix}, this result may be generalized to cases where $h^1(C,\mL)>2$ only if one poses some non-degeneracy condition on the vector bundles. In this paper, we propose a reasonable non-degeneracy condition to realize this generalization. Our main result is the following:
\begin{thm}\label{thm2}
For a Petri general curve $C$, denote by $\GG^{\text{gg}}(2,\mL,k)$ the open sub-stack of $\GG(2,\mL,k)$ parametrizing data whose underlying vector bundle is generically globally-generated by elements in the $k$-dimensional subspace $V$ of sections. Then, every irreducible component of $\GG^{\text{gg}}(2,\mL,k)$ such that $h^1(\mL)\ge m$ has dimension at least 
$$\rho^2_{\mL}(k,g):=\rho(2,d,k,g)-g+m\binom{k}{2}.$$
\end{thm}
It is worth mentioning that as a generalization of the result of Bertram, Feinberg and Mukai, our theorem also partially confirms a recent conjecture by Grzegorczyk and Newstead on coherent systems (see section \ref{mainsection}).
\subsection*{Acknowledgements.}The author would like to thank his advisor, Brian Osserman, for introducing this problem and his
tireless instruction. In particular, the author would like to thank him for providing an interesting argument to clarify one non-example in the paper.  (See Section \ref{non}.) 
\section{Notation and Terminology}
Hereafter, we work over some base field $K$.

The object of interest in this note is the multiply symplectic Grassmannian:
\begin{defn}
Let $E$ be an $n$-dimensional vector space over an algebraically closed field $K$. Suppose $\form_1,...,\form_m$ are $m$
independent symplectic forms on $E$. The multiply symplectic Grassmannian, $MSG(k,E;\form_1,...,\form_m)$, is defined to be the
sub-scheme of $Gr(k,E)$ parametrizing all $k$-dimensional subspaces $V$ of $E$ which are simultaneously isotropic with respect
to $\form_1,...,\form_m$.
\end{defn}
\begin{rem}
  Since $V$ is isotropic with respect to $\form_t$, for any $v\in V$, $\lag v,\cdot \rag_t$ naturally induces an element in
  $(E/V)^*$. Consequently, fixing $v,v'\in V$, the map $\ho(V,E/V)\to K:g\mapsto \lag v,g(v')\rag_t$ defines an element in
  $\ho(V,E/V)^*$.
\end{rem}
As mentioned earlier, we seek to study tangent spaces, $T_VMSG(k,E;\form_1,...,\form_m)$, of $MSG(k,E;\form_1,...,\form_m)$. So
we first give a concrete description of them.
\begin{lem}
  Let $V$ be a subspace of $E$ which is isotropic with respect to a symplectic form $\form$ over $E$. Then, $T_VMSG(k,E;\form)$
  can be identified with $$\{f\in \ho(V,E/V)|\langle\cdot,f(\cdot)\rangle\in S^2(V)^*\}.$$
\end{lem}
\begin{proof}
  The tangent space, $T_VGr(k,E)$, of the Grassmannian $Gr(k,E)$ can be identified with $\ho(V,E/V)$.

  By first-order infinitesimal calculation, $f$ is in the tangent space of the symplectic Grassmannian if and only if $\fa v_1,v_2\in V$,

  $\begin{aligned}
   0&=\langle v_1+\ep f(v_1),v_2+\ep f(v_2)\rangle=\langle v_1,\ep f(v_2)\rangle+\langle \ep f(v_1),v_2\rangle\\
    &=\langle v_1,\ep f(v_2)\rangle-\langle v_2,\ep f(v_1)\rangle\\
    &=\ep (\langle v_1,f(v_2)\rangle-\langle v_2,f(v_1)\rangle)
  \end{aligned}$

  This concludes the proof.
\end{proof}
\begin{cor}
The tangent space at $[V]$ of $MSG(k,E;\form_1,...,\form_m)$ can be identified with the subspace of elements in $\ho(V,E/V)$
which are simultaneously symmetric with respect to $\form_1,...,\form_m$.
\end{cor}
The condition that $f\in \ho(V,E/V)$ is symmetric with respect to $\form_t$ imposes $\binom{k}{2}$ many linear conditions on
$\ho(V,E/V)$. Fixing a basis $v_1,...,v_k$ for $V$ the conditions can be described as follows:
$$\langle v_i,f(v_j)\rangle_t-\langle v_j,f(v_i)\rangle_t=0,\fa i<j.$$
Thus, the dimension of $T_VMSG(k,E;\form_1,...,\form_m)$ is at least $k(n-k)-m\binom{k}{2}$.

Define $g^t_{i,j}=\langle v_i,(\cdot)(v_j)\rangle_t-\langle v_j,(\cdot)(v_i)\rangle_t\in\ho(V,E/V)^*$, $\fa i<j$. To study when
the above expected dimension is obtained is then tantamount to understanding when $g^t_{i,j}$ are independent. We shall
describe and discuss a necessary and sufficient condition in the next section.

Throughout the note, $\ff$ will denote a set of finitely many symplectic forms, and $m$ always denotes the number of forms in
it.
\section{A New Perspective on the Problem}
Recall the definition of $g^t_{i,j}$ from previous section. Denote $\Omega=\s(\form_1,..,\form_m)$. The set of elements
$\{g^t_{i,j}|t=1,...,m;i<j\}$ in $\ho(V,E/V)^*$ naturally corresponds to a linear map $\Phi:\Lambda^2V\ot\Omega\to
\ho(V,E/V)^*$ sending $(v\wedge v')\ot\form$ to $\lag v,(\cdot)(v')\rag-\lag v',(\cdot)(v)\rag$ and extending by linearity.
Obviously, $T_VMSG(k,E;\ff)$ is of expected dimension $k(n-k)-m\binom{k}{2}$ if and only if $\Phi$ is injective. We give an
equivalent formulation of this condition:
\begin{prop}\label{main3}
  Define $j_V:V\otimes\Omega\to(E/V)^*:v\otimes\form\mapsto\lag v,\rag$. $\Phi:\Lambda^2V\ot\Omega\to \ho(V,E/V)^*$ fails to be
  injective if and only if $\exists \psi:V^*\to V\ot\Omega\cong \ho(V^*,\Omega)$ not zero such that $\psi(v)(v)=0$
  for all $v\in V^*$ and $j_V\circ\psi=0$.
\end{prop}
\begin{proof}
First, identify $\Lambda^2V$ with the subspace of elements in $(V^*\ot V^*)^*$ sending $v\ot v$ to zero (for all $v\in  V^*$) via the map $v_1\wedge v_2\mapsto (f_1\ot f_2\mapsto f_1(v_1)f_2(v_2)-f_1(v_2)f_2(v_1))$. Under the natural isomorphism $(V^*\ot V^*)^*\cong V\ot V\cong \ho(V^*,V)$, we then have the following identification:
  $$\begin{aligned}
    &\Lambda^2V\ot\Omega\cong\{f\in\ho(V^*,V)|v(f(v))=0,\fa v\in V^*\}\otimes\Omega\\
    &\cong\{\psi\in\ho(V^*,V\otimes\Omega)|\psi(v)(v)=0\in\Omega,\fa v\in V^*\}
  \end{aligned}$$
  
(In the last step, we also identify $\ho(V^*,V\otimes\Omega)$ with $\ho(V^*,\ho(V^*,\Omega))$.)

With this identification, $\Phi$ can be thought of as a map sending $\psi$ to $j_V\circ\psi$. Here, we identify the following vector spaces: $\ho(V^*,(E/V)^*)\cong V^{**}\ot(E/V)^*\cong (V^*\ot(E/V))^*\cong \ho(V,E/V)^*$.

Hence, the result follows.
\end{proof}

We are now ready to prove Theorem \ref{prop:tech} as a corollary of Proposition \ref{main3}.
\begin{proof}[Proof of Theorem \ref{prop:tech}]
The ``if'' direction is trivial, since if there exists a 2-dimensional subspace $V'$ of $V$ and a non-zero form $\form\in\Omega$ such that $\form|_{V'\times E/V}=0$, take two independent vectors $v,v'$ in $V'$, then $v\wedge v'\ot\form$ is in the kernel of $\Phi$ (as in \ref{main3}).

We now prove the ``only if'' direction. In view of Proposition \ref{main3}, we will show the following statement: if there exists $\psi:V^*\to \ho(V^*,\Omega)$ not zero such that $\psi(v)(v)=0$ for all $v\in V^*$ and $j_V\circ\psi=0$, then there exists a 2-dimensional subspace $V'$ of $V$ such that $\form|_{V'\times E/V}=0$.

Consider a non-zero linear map $\psi:V^*\to
V\ot\Omega$ such that $j_V\circ\psi=0$. Fixing a basis of $V$, $\psi$ can be presented as $M_1\ot\form_1+M_2\ot\form_2$, where
$M_1,M_2$ are two anti-symmetric matrices and $\form_1,\form_2$ are the two symplectic forms spanning $\Omega$.

By Thm XV.8.1 in \cite{L}, up to a choice of basis $\{v_1,...,v_k\}$ for $V$, one can assume that $\rk(M_1)=2m\ge2$
($\psi\neq0$) and $M_1$ is block diagonal, where diagonal blocks are either of the form $\begin{bmatrix}0&1\\-1&0\end{bmatrix}$
or are zeros.

Denote $M_1=(a_{i,j}),M_2=(b_{i,j})$. Then, $j_V\circ\psi(v_i^*)=\lag v^1_i,\cdot\rag_1+\lag v^2_i,\cdot\rag_2$, where $v^1_i=\sum_j
a_{i,j}v_j,v^2_i=\sum_j b_{i,j}v_j$.

If the $i$-th row of $M_1$ is zero, so is the $i$-th row of $M_2$; otherwise one gets $\lag v^2_i,\cdot\rag_2=0$ for some
non-zero $v^2_i$, which violates the assumption that $\form_2$ is non-degenerate. Therefore, by replacing $V$ with $\s(\{v_i\}_{i=1}^{2m})$ (such that the $i$-th row of $M_1$ is non-zero) if necessary, one can assume that $M_1,M_2$
are both non-degenerate of the same (even) rank. In this case, there is an element $GL(V)$ sending $v^1_i$ to $v^2_i$. Under
the fixed choice of basis, this linear map is presented by the matrix $M_2M_1^{-1}$. Let $d\neq0$ be an eigenvalue of $M_2M_1^{-1}$.
We claim the eigenspace associated to $d$ is even-dimensional.

Indeed, to solve $(dI_k-M_2M_1^{-1})\vec{X}=0$ is equivalent to solving $(dM_2^{-1}-M_1^{-1})\vec{X}=0$, where
$dM_2^{-1}-M_1^{-1}$ is a singular $k\times k$ anti-symmetric matrix. By Thm XV.8.1 in \cite{L} again, one can conclude $dM_2^{-1}-M_1^{-1}$ has even rank. Since we
started with the assumption $\dim V$ is even, one can conclude that the solution space of $(dM_2^{-1}-M_1^{-1})\vec{X}=0$ is
even-dimensional. This proves the claim.

Therefore, there exists $0\neq d\in K$ such that $\form_1+d\form_2$ vanishes on $V'\times E/V$, where $V'\subset V$ is some
non-zero even-dimensional subspace of $V$.

We have thus proved Theorem \ref{prop:tech}.
\end{proof}

In order to study rank two Brill-Noether loci with special determinant in general, we further re-formulate the condition in \ref{main3}
as follows:
\begin{prop}\label{main2}
Let $\Omega$ be a $K$-vector space spanned by $m$ independent symplectic forms $\form_1,...,\form_m$ on an $n$-dimensional
vector space $E$. Let $V$ be a $k$-dimensional subspace of $E$ simultaneously isotropic with respect to $\form_1,...,\form_m$. Then, there exists a non-zero map $\psi:V^*\to V\ot\Omega$ such that
$\psi(v)(v)=0$ for all $v\in V^*$ and $j_V\circ\psi=0$ ($j_V$ as in \ref{main3}), if and only if there exists a basis $\{v_1,...,v_k\}$ of $V$ and $\form_{ij}\in\Omega$ (for all $i<j$) not all zero such that 
\begin{equation}\label{m2}
W=\s(\{\sum_{i:j>i}(-v_i\ot\form_{ij})+\sum_{i:j<i}(v_i\ot\form_{ji})|j=1,...,k\}),\end{equation}
is a subspace of $\ker(j_V)$. 
\end{prop}
\begin{proof}
  Suppose $j_V\circ \psi=0$. One can write $\psi=\dis\sum_{s=1}^m \psi_s\ot\form_s$, where $\psi_s\in\ho(V^*,V)$ for all $s$ and $\psi_s(v)(v)=0$ for all $v\in V^*$. (We identify $V$ with $V^{**}$ in the natural way.) As in \ref{main3}, each $\psi_s$ can be identified with an alternating form on $V$. Fix a basis $\{v_1,...,v_k\}$ of $V$, one may represent $\psi_s$ by an alternating matrix $(a^i_{j,s})$. In particular, $\fa f\in V^*$, $\psi_s(f)=\sum_{j=1}^kf(\dis\sum_{i<j}-a^i_{j,s}v_i+\dis\sum_{i>j}a^j_{i,s}v_i)v_j$.
  
For all $i<j$, let $\form_{ij}=a^i_{j,1}\form_1+...+a^i_{j,m}\form_m$. One can then check 
$$\sum_{i:j>i}(-v_i\ot\form_{ij})+\sum_{i:j<i}(v_i\ot\form_{ji})=\sum_{s=1}^m(\dis\sum_{i<j}-a^i_{j,s}v_i+\dis\sum_{i>j}a^j_{i,s}v_i)\ot\form_s$$
is inside $\ker(j_V)$, for $j=1,...,k$.
 
The ``if'' part of the statement can be verified by reversing the direction of the above argument.
\end{proof}
\section{Application to the Rank Two Brill-Noether Problem}\label{mainsection}
We now apply the result in the previous section to the rank two Brill-Noether problem with fixed special determinant.

The connection between the the local geometry of the multiply symplectic Grassmannian and rank two Brill-Noether theory was established in \cite{Ofix}. Let $\mE$ be a rank two vector bundle on a smooth curve $C$ with determinant $\mL$ such that $h^1(C,\mL)=m>0$. Let $\phi_1,...,\phi_m$ be $m$ linearly independent maps $\mL\to\omega_C$ spanning $H^0(C,\omega_C\ot\mL^{-1})\cong H^1(C,\mL)^*$. They induce $m$ symplectic forms on $\ti{\mE}:=\mE(D)/\mE(-D)$, where $D$ is a certain ample divisor:
\begin{lem}
  Given $\phi_i:\mL\to\omega_C$ ($i=1,...,m$), $m$ linearly independent morphisms, suppose $D$ is an ample divisor supported away from the vanishing locus of $\phi_i$. Denote $\ti{\mE}:=\mE(D)/\mE(-D)$. Then,
  $$\lag,\rag_i:H^0(\ti{\mE})\times H^0(\ti{\mE})\to \mO_C:\lag s_1,s_2\rag_i=\sum_{P\in C}\res_P\phi_i(\ti{s}_{1,P}\wedge\ti{s}_{2,P})$$
  defines $m$ linearly independent symplectic forms on $H^0(\ti{\mE})$, where $\ti{s}_{i,P}$ is a representative in $\mE(D)$ of $s_i$ near $P$.

  When $D$ is sufficiently ample, $H^0(\mE(D))$ is isotropic with respect to $\form_i$.
\end{lem}
\begin{proof}
  See Lemma 5.1, 5.2 in \cite{Ofix}.
\end{proof}
Hereafter, fix some $\phi\in H^0(C,\mL^*\ot\omega)$ and denote $Z=\dv(\phi)$. Also fix a reduced effective divisor $D=P_1+P_2+...+P_{\deg(D)}$ such that $\phi$ does not vanish on any $P_t$ and that $\deg(D)>2g-2-d+M$, where $M$ is some upper bound for degrees of sub-line bundles of $\mE$.

By Proposition \ref{main2}, for any $k$-dimensional subspace $V\subset H^0(\mE)$,
the tangent space $T_V\M(k,H^0(C,\ti{\mE}))$ has codimension less than $m\binom{k}{2}$ inside $T_V\Gr(k,H^0(C,\ti{\mE}))$ if and only if there exists a basis $s_1,...,s_k$ of $V$ and $\form_{ij}\in\Omega$ not all zero such that
$$-\sum_{i:j>i}(\lag s_i,\cdot\rag_{ij})+\sum_{i:j<i}(\lag s_i,\cdot\rag_{ji})=0\in H^0(C,\ti{\mE})^*,\text{ for }j=1,...,k.$$
Written in terms of residues, one gets:
\begin{equation}\label{res1}-\sum_{i:j>i}\sum_{P\in C}\res_P\phi_{ij}(s_{i}\wedge\cdot)+\sum_{i:j<i}\sum_{P\in C}
\res_P\phi_{ji}(s_{i}\wedge\cdot)=0\in H^0(C,\ti{\mE})^*.\end{equation}
Here, $\phi_{ij}$ are elements in $H^0(C,\mL^*\ot\omega)=\ho(\mL,\omega)$ which are not all zero.

Write $\phi_{ij}=f_{ij}\phi$, for some rational function $f_{ij}$ on $C$ such that $(f_{ij})+Z\ge0$. Equation \eqref{res1} can be further written as
\begin{equation}\label{res2}\sum_{P\in C}\res_P\phi((-\sum_{i:j>i}f_{ij}s_i+\sum_{i:j<i}f_{ji}s_i)\wedge\cdot)=0\in H^0(C,\ti{\mE})^*,
j=1,2,...,k.\footnote{$\phi$ naturally induces a morphism $\mL(Z+D)\to\omega(Z+D)$.}\end{equation}

Notice that $-\sum_{i:j>i}f_{ij}s_i+\sum_{i:j<i}f_{ji}s_i\in\Gamma(C,\mE(Z))$ and by our assumption $\supp(Z)\cap\supp(D)=\emptyset$ for all $P_t\in\supp(D)$. Therefore, $-\sum_{i:j>i}f_{ij}s_i+\sum_{i:j<i}f_{ji}s_i$ is regular at $P_t$, for all $t$. Consider $s'\in H^0(C,\ti{\mE})$ only supported at $P_t$. Equation \eqref{res2} then
implies
\begin{equation}\label{res3}\res_{P_t}\phi((-\sum_{i:j>i}f_{ij}s_{i}+\sum_{i:j<i}f_{ji}s_{i})\wedge s')=0, j=1,2,...,k\end{equation}
One then concludes that $-\sum_{i:j>i}f_{ij}s_{i}+\sum_{i:j<i}f_{ji}s_{i}$ vanishes at all $P_t$, i.e.
$$-\sum_{i:j>i}f_{ij}s_{i}+\sum_{i:j<i}f_{ji}s_{i}\in\Gamma(C,\mE(Z-D)).$$
As a section of $\mE(Z)$, if $-\sum_{i:j>i}f_{ij}s_{i}+\sum_{i:j<i}f_{ji}s_{i}$ is non-zero, it has at most $2g-2-d+M$ many zeros, where $M$ is the upper bound for degrees of sub-line bundles of $\mE$ chosen before. Since $\deg(D)>2g-2-d+M$ and $-\sum_{i:j>i}f_{ij}s_i+\sum_{i:j<i}f_{ji}s_i$ vanishes on $\supp(D)$, it follows that
\begin{equation}\label{main5}
  -\sum_{i:j>i}f_{ij}s_i+\sum_{i:j<i}f_{ji}s_i=0.
\end{equation}
We have thus arrived at the following result:
\begin{prop}\label{main6}
Denote
$$m:H^0(C,\mL^*\ot\omega)\ot H^0(C,\mE)\to H^0(C,\mL^*\ot\mE\ot\omega)\cong H^0(C,\mE^*\ot\omega)$$
to be the natural multiplication map. Then, \\
$T_V\M(k,H^0(C,\ti{\mE}))$ has codimension less than $m\binom{k}{2}$ inside $T_V\Gr(k,H^0(C,\ti{\mE}))$ if and only if there exists a basis $\{s_1,..,s_k\}$ of $V$ and $\phi_{ij}\in H^0(C,\mL^*\ot\omega)$ ($\fa i<j$) not all zero such that $-\sum_{i:j>i}\phi_{ij}\ot s_i+\sum_{i:j<i}\phi_{ji}\ot s_i\in\ker(m)$ for all $j$.
\end{prop}
Therefore, we would like to draw useful information about $\ker(m)$.

The following observation is crucial for our analysis: the multiplication map $m$ is related to the classical Petri map via the following commutative diagram:
$$\xymatrix{H^0(C,\mL^*\ot\omega)\ot H^0(C,\mE)\ot H^0(C,\mE)\ar[r]^-{\Id\ot\det}\ar[d]_{m\ot\Id}&H^0(C,\mL^*\ot\omega)\ot H^0(C,\mL)\ar[d]^{\text{Petri}}\\H^0(C,\mE^*\ot\omega)\ot H^0(C,\mE)\ar[r]& H^0(C,\omega)}$$
Let $\sum f_i\ot s_i$ be an element in $\ker(m)$, then for any $t\in H^0(C,\mE)$, $(\sum f_i\ot s_i)\ot t$ is sent to zero in $H^0(C,\omega)$ along either direction of the commutative diagram. In particular, $\sum f_i\ot (s_i\wedge t)$ is in the kernel of the Petri map. By the Gieseker-Petri Theorem (\cite{Gieseker}), over a general curve $C$, the Petri map is injective and hence  $\sum f_i\ot (s_i\wedge t)=0\in H^0(C,\mL^*\ot\omega)\ot H^0(C,\mL)$. Without loss of generality, one can assume $f_i$ are independent in $H^0(\mL^*\ot\omega)$ in which case, $s_i\wedge t=0$ for all $i$. Since this holds for all $t\in H^0(\mE)$, one can conclude that there exists a sub-line bundle $\mF$ of $\mE$ such that $H^0(\mF)=H^0(\mE)$.
We have thus arrived at the following proposition:
\begin{prop}\label{main-basic}
Let $\mE$ be a rank two vector bundle with special determinant on a Petri general curve $C$. If the sections of $\mE$ are not colinear, then the multiplication map
$$m:H^0(C,\mL^*\ot\omega)\ot H^0(C,\mE)\to H^0(C,\mL^*\ot\mE\ot\omega)\cong H^0(C,\mE^*\ot\omega)$$
is injective.
\end{prop}
We are now ready to prove our main result:
\begin{proof}[Proof of Theorem \ref{thm2}]
This follows from our previous analysis and several fundamental results in \cite{Ofix}, \cite{Osp2}.

Adopting the framework in \cite{Ofix}, let $\cu$ be any open sub-stack of the moduli stack of tank two vector bundles with fixed determinant $\mL$ over $C$ and denote $\mE_{\cu}$ to be the universal bundle over $\cu$. Fix a reduced effective divisor $D$ as before and denote $D'$ to be the pull-back of $D$ to $C\times\cu$. $m$ independent elements $\phi_1,...,\phi_m\in H^0(C,\mL^*\ot\omega)$ induce $m$ independent symplectic form on $p_*(\mE_{\cu}(D')/\mE_{\cu}(-D'))$, where $p:C\times\cu\to\cu$ is the natural projection. Then, the open sub-stack $\GG_{\cu}(2,\mL,k)$ of $\GG(2,\mL,k)$ (see \ref{thm1}) parametrizing data for which the underlying vector bundle lies in $\cu$ can be identified with the closed-sub-stack of the relative Grassmannian $\GG(k,p_*(\mE_{\cu}(D')/\mE_{\cu}(-D')))$ parametrizing sub-bundles contained in both $p_*(\mE_{\cu}/\mE_{\cu}(-D'))$ and $p_*(\mE_{\cu}(D'))$, which are isotropic with respect to the $m$ symplectic forms. 

For a bundle $\mE$ over $C$ which is contained in $\cu$, Proposition \ref{main-basic} is a point-wise result for $\GG_{\cu}(2,\mL,k)$. We can then apply a codimension counting result (Proposition 2.4) in \cite{Osp2} (and let $\cu$ run through all open sub-stacks of finite type of the moduli of vector bundles) to conclude Theorem \ref{thm2}.
\end{proof}
Now assume the vector bundle $\mE$ is stable and has $k$ sections. By the Classical Brill-Noether Theorem,  when $k$ is sufficiently large, no sub-line bundle of $\mE$ can have $k$ sections. In particular, $\mE$ is generically globally generated. Thus, one gets the following corollary:
\begin{cor}
For a Petri general curve $C$, suppose $g+k(1-g+\lfloor{d\over2}\rfloor)-k^2<0$. Then, every irreducible component of the stable locus of $\GG(2,\mL,k)$ such that $h^1(\mL)=m\ge 1$ has dimension at least
$$\rho^2_{\mL}(k,g):=\rho(2,d,k,g)-g+m\binom{k}{2}.$$
When $m\le 2$, the result holds without the inequality condition.
\end{cor}
In addition, we would like to point out the relation between Theorem \ref{thm2} and a conjecture of Grzegorczyk and Newstead. 

In \cite{CSFD},  Grzegorczyk and Newstead made the following conjecture:
\begin{conj}[Conjecture 2.3 in \cite{CSFD}]
Denote $G(\aaa;r,\mL, k)$ to be the moduli space of $\aaa$-stable coherent systems of type $(r, d, k)$. Then,
every irreducible component $X$ of $G(\aaa;r,\mL, k)$ has dimension
$$\dim X \ge\rho(r, d, k,g)-g+\binom{k}{r}\cdot h^1(\mL)$$
\end{conj}
It is well-known that over a Petri general curve $\aaa$-stability implies generically globally generated when $\aaa$ sufficiently large (see Theorem 3.1 in \cite{CSB}). Translating our result into the language of coherent systems, one sees that our result verifies this conjecture in rank two for large $\aaa$ values.
\section{A Non-example}\label{non}
Lastly, we give an example to show how the Petri generality assumption rules out cases in which the multiplication map $m$ fails to be injective.    

Notice that if $s,fs\in H^0(C,\mL^*\ot\omega)$ and $s',fs',\in H^0(C,\mE)$, where $f\in K(C)$, then $s\ot (fs')-(fs)\ot s'\in \ker(m)$. We therefore  have the following non-example:

\begin{lem}\label{non1}
  The condition in \eqref{main5} holds for three independent sections $s_1,s_2,s_3\in\Gamma(C,\mE)$ and rational functions $f_{ij}$ not all zero ($1\le i,j\le 3$, $i<j$) (such that for some $\phi\in\Gamma(C,\mL^*\ot\omega)$, $f_{ij}\phi\in\Gamma(C,\mL^*\ot\omega)$) if and only if there exists an invertible subsheaf $\mL'$ of $\mE$ such that $h^0(\mL')\ge 3$ and $\mL'$ injects into $\mL^*\ot\omega$.
\end{lem}
\begin{proof}
Given $\mL'$ to be such a subsheaf, let $s_1,s_2,s_3$ be three sections of $\mL'$ such that $s_2=fs_1,s_3=gs_1$. Denote $\phi\in H^0(C,\mL^*\ot\omega)$ to be the image of $s_1$ under the injection $H^0(C,\mL')\to H^0(C,\mL^*\ot\omega)$ induced by the injection $\mL'\to\mL^*\ot\omega$. Denote $\phi_2=f\cdot\phi,\phi_3=g\cdot\phi$. It is not hard to see that
\begin{equation}
\left\{\begin{aligned}
&0+(-1)s_2+fs_1=0\\
&(-1)s_3+0+gs_1=0\\
&(-f)s_3+(-g)s_2+0=0
\end{aligned}\right.
\end{equation}
in other words, the condition in \eqref{main5} holds.

Conversely, assume the condition in \eqref{main5} holds for $s_1,s_2,s_3\in\Gamma(C,\mE)$ and $\phi,f_{ij}\phi\in\Gamma(C,\mL^*\ot\omega)$ ($1\le i,j\le 3$, $i<j$). Clearly, \eqref{main5} implies $s_1,s_2,s_3$ are co-linear. Suppose $s_1=fs_3,s_2=gs_3$ and without loss of generality one can assume $f_{13}=-f,f_{12}=-g$. Let $D_1=\dv(s_3)$, $D_2=\dv(\phi)$. Then, $(f)+D_i,(g)+D_i\ge0$, for $i=1,2$. Write $D_i=\sum n^i_P P$ and define $Y=\sum\min\{n^1_P,n^2_P\}P$. $Y$ is an effective divisor and $(f)+Y\ge0, (g)+Y\ge0$. In particular, $\mO(Y)$ is an invertible subsheaf of both $\mE$ and $\mL^*\ot\omega$, and $h^0(\mO(Y))\ge3$.
\end{proof}

\subsection*{A remark on the non-example} The following argument by Brian Osserman explains how Petri generality precludes the non-example stated in Lemma \ref{non1}. Suppose the curve $C$ is Petri general. Given some $\mL'$ as in \ref{non1}, one obtains a pencil with an injection $0\to\mL'\to\mL^*\ot\omega$. Suppose $H^0(\mL')\neq H^0(\mE)$. One gets a short exact sequence $0\to\mL'\to \mE\to Q\to0$ such that $h^0(Q)\ge 1$. Since $Q\cong\mL'^*\ot\mL$, this implies there is an injection $0\to\mL'\to\mL$. Taking a product of the two morphisms, one obtains an injection $0\to\mL'^{\ot2}\to\omega$, i.e. $h^1(\mL'^{\ot2})\ge1$. This violates the Petri generality assumption.

\bibliographystyle{amsalpha}
\bibliography{myrefs}

\providecommand{\bysame}{\leavevmode\hbox to3em{\hrulefill}\thinspace}
\providecommand{\MR}{\relax\ifhmode\unskip\space\fi MR }
\providecommand{\MRhref}[2]{%
  \href{http://www.ams.org/mathscinet-getitem?mr=#1}{#2}
}
\providecommand{\href}[2]{#2}
\begin{thebibliography}{BBPN08}

\bibitem[BBPN08]{CSB}
U.~N. Bhosle, L.~Brambila-Paz, and P.~Newstead, \emph{On {C}oherent {S}ystems
  of {T}ype $(n,d,n+1)$ on petri curves}, Manuscripta Mathematica \textbf{126
  (4)} (2008), 409--441.

\bibitem[BF98]{Ber}
A.~Bertram and B.~Feinberg, \emph{On {S}table {R}ank {T}wo {B}undles with
  {C}anonical {D}eterminant and {M}any {S}ections}, Algebraic Geometry, Papers
  for Europroj Conferences in Catania and Barcelona (1998), 259--269.

\bibitem[Gie82]{Gieseker}
D.~Gieseker, \emph{Stable {C}urves and {S}pecial {D}ivisors}, Invent. Math.
  \textbf{66, no.2} (1982), 251--275.

\bibitem[GN14]{CSFD}
I.~Grzegorczyk and P.~Newstead, \emph{On {C}oherent {S}ystems with {F}ixed
  {D}eterminant}, Internat. J. math (2014).

\bibitem[Lan02]{L}
Serge Lang, \emph{Algebra}, Springer, 2002.

\bibitem[LNP13]{LNP}
H.~Lange, P.~Newstead, and S.~Park, \emph{Non-emptiness of {B}rill-{N}oether
  {L}oci in ${M}(2,k)$}, arXiv:1311.5007 (2013).

\bibitem[Muk95]{MCan}
S.~Mukai, \emph{Vector {B}undles and {B}rill-{N}oether {T}heory}, Current
  Topics in Complex Algebraic Geometry \textbf{MSRI Publications, vol. 28}
  (1995), 145--158.

\bibitem[Oss13a]{Ofix}
B.~Osserman, \emph{Brill-{N}oether {L}oci with {F}ixed {D}eterminant in {R}ank
  2}, International Journal of Mathematics \textbf{24, no. 13} (2013).

\bibitem[Oss13b]{Osp2}
\bysame, \emph{Special {D}eterminants in {H}igher-rank {B}rill-{N}oether
  {T}heory}, International Journal of Mathematics \textbf{24, no.11} (2013),
  1350084.

\bibitem[TiB04]{M1}
M.~Teixidor~i Bigas, \emph{Rank {T}wo {V}ector {B}undles with {C}anonical
  {D}eterminant}, Math. Nachr. \textbf{265} (2004), 100--106.

\bibitem[Zha14]{Zhang1}
N.~Zhang, \emph{Towards the {B}ertram-{F}einberg-{M}ukai {C}onjecture},
  arXiv:1409.0971 (2014).

\end{thebibliography}
\end{document}